\documentclass[12pt,a4paper,oneside,titlepage]{article}
\usepackage{verbatim} % For multiline commenting with \begin{comment} and \end{comment}.
\usepackage{amsmath,amsfonts,amssymb,amsthm,amscd}
\newtheorem{Theorem}{Theorem}

\newtheorem{Remark}[Theorem]{Remark}
\newtheorem{Lemma}[Theorem]{Lemma}
\newtheorem{Corollary}[Theorem]{Corollary}

\newtheorem{Example}[Theorem]{Example}

\newtheorem{Question}[Theorem]{Question}

\makeatletter
\newcommand*{\house}[1]{%
	\mathord{%
		\mathpalette\@house{#1}%
	}%
}
\newcommand*{\@house}[2]{%
	\dimen@=\fontdimen8 %
	\ifx#1\scriptscriptstyle\scriptscriptfont
	\else\ifx#1\scriptstyle\scriptfont
	\else\textfont\fi\fi
	3 %
	\sbox0{%
		$#1%
		\vrule width\dimen@\relax
		\overline{%
			\kern2\dimen@
			\begingroup % to keep changes of \dimen@ in #2 local
			#2%
			\endgroup
			\kern2\dimen@
		}%
		\vrule width\dimen@\relax
		\mathsurround=1.5\dimen@ % outside margin
		$%
	}%
	\ht0=\dimexpr\ht0-\dimen@\relax
	\dp0=\dimexpr\dp0+2\dimen@\relax
	\vbox{%
		\kern\dimen@ % reinsert previously removed space
		\copy0 %
	}%
}

\title{Linear independence of continued fractions with algebraic terms
\footnotetext{AMS Class: 11J72, 11J70.}
\footnotetext{Key words and phrases: continued fraction; algebraic number; linear independence} }
\author{Jaroslav Han\v{c}l, Mathias L. Laursen, and Jitu Berhanu Leta}

\begin{document}
\maketitle
\date{}

\begin{abstract}
We give conditions on sequences of positive algebraic numbers $\{a_{n,j}\}_{n=1}^\infty$, $j=1,\dots ,M$ and number field $\mathbb K$    to ensure that the numbers defined by the continued fractions $[0;a_{1,j},a_{2,j},\dots ]$, $j=1,\dots ,M$ and $1$ are linearly independent over $\mathbb K$.
\end{abstract}

\section{Introduction}

Following Erd\H os \cite{Erdos}, Davenport and Roth \cite{dav} we prove:  
\begin{Theorem}
\label{hanluc}
   Let
  $\{a_n\}_{n=1}^\infty$ 
  be a non-decreasing sequence of positive integers such that
    $\lim\limits_{n\rightarrow\infty} a_n^{\frac{1}{3^n}} =
    \infty$ and $\{ p_n\}_{n=1}^\infty$ be the increasing sequence of all primes. 
 Then the continued fractions 
 $$ \Big[0;\Big(1+\frac {p_1}1\Big)\sqrt2,\dots ,\Big(1+\frac {p_n}n\Big)\sqrt 2,\dots\Big], \  \ \Big[0;a_1\frac {p_1}1\sqrt 2,\dots ,a_n\frac {p_n}n\sqrt 2,\dots\Big]$$
 and number $1$ are linearly independent over $\mathbb Q(\sqrt 2)$ particularly  over $\mathbb Q$ . 
\end{Theorem}
This is an immediate consequence of Theorem \ref{hanclJitu1.T1}, which will be introduced in the chapter Main Results. 
The results presented in this paper have some history. Forty years ago Davenport and Roth in \cite{dav} proved that the continued fraction $[a_1; a_2,\dots]$, where $a_1,a_2,\dots$ are positive integers satisfying 
$\limsup_{n\to\infty}\left((\log \log {a_n})\frac{\sqrt{\log n}}{n}\right)=\infty,$
is a transcedental number. Han\v cl \cite{hs2} found some criteria for continued fractions to be linearly independent. Andersen and Kristensen \cite{and}  come with special conditions on continued fractions consisting of algebraic integers to be irrational or transcedental numbers. The generalization of transcendence is algebraic independence and there are several results concerning the algebraic independence of continued fractions, see, for instance, \cite{bun}, \cite{hs3}  or \cite{hkn}. In 1975, Erd\H os \cite{Erdos} proved that if  $\{a_n\}_{n=1}^\infty$ is a non-decreasing sequence of positive integers such that 
$\lim_{n\to\infty} a_n^{\frac 1{2^n}}=\infty$ then the number $\sum_{n=1}^\infty \frac 1{a_n}$ is irrational. Later, in 1991, Han\v cl \cite{han0} proved that if $\{a_n\}_{n=1}^\infty$ is a sequence of positive real numbers such that $a_n\leq 2^{\frac 1{n^2}2^n}$ holds for any positive integer $n$, then there exists a sequence $\{c_n\}_{n=1}^\infty$ of positive integers such that the number $\sum_{n=1}^\infty \frac 1{c_na_n}$ is rational. Rucki \cite{ruc} established a criterion for the sums of reciprocals of a sequence of natural numbers to be irrational. Gen\v cev \cite{gen1} obtained some irrationality results with the help of special transformations. Then Han\v cl and Sobkov\' a  \cite{hs1} established the linear independence of the sums of certain infinite series.  Using Pad\' e approximation Matala-aho and Zudilin \cite{mz} obtained some interesting results in irrationality of infinite series. Recently, Han\v cl and Kolouch \cite{hk} gave a criterion for infinite products of rational numbers to be irrational. A nice review of these results can be found in \cite{karf} and \cite{kn}. 

Our results are of a quite general character and written in the spirit of Erd\" os.  This method was later develop by  Han\v cl and Sobkov\' a \cite{hs}, see also \cite{hpks}. We do not for instance require that the elements of $\{ a_n \} _{n=1}^{\infty}$  be approximable by the elements of a finite union of power sequences or be associated with any differential equation as in the method of K. Mahler, for which the reader is referred to K. Nishioka's book \cite{nish}. Let us mention also \cite{nish1}.

The main result of this paper is Theorem \ref{hanclJitu1.T1}. Many consequences and examples of this theorem can be found in the chapter Main Results.  Lemma \ref{hanclJitu1.l1} deals with the conditions which guarantee that the ratio of two linear recurrences tends to infinity.

\section{Notations}
We use the standard notation: $\mathbb N$ and $\mathbb Z$ the set of non-negative integers and integers, respectively. 
For a positive real number $x$  the expression $\log x$, $\log_2 x$ and $\pi(x)$ denotes the natural logarithm of  $x$, the logarithm base $2$ of $x$ and number of primes less than or equal to $x$, respectively. If $x$ is positive integer then $d(x)$ denotes the number of divisors of $x$.

Let $\frac {p_n}{q_n}=[a_0;a_1,a_2,\cdots ,a_n]$ be the $n$-th partial fraction of the real number  $a=[a_0;a_1,a_2,\cdots ]$. We have 
$$p_0=a_0,\quad q_0=1,\quad p_1=a_1a_0+1,\quad q_1=a_1,\quad
p_{n+2}=a_{n+2}p_{n+1}+p_n, \quad $$ 
$$q_{n+2}=a_{n+2}q_{n+1}+q_n,\quad  
q_{n+1}p_n-p_{n+1}q_n=(-1)^{n+1},\quad 
$$ 
\begin{align*}
	a&=[a_0;a_1,a_2,\cdots ]=\big[a_0;a_1,a_2,\cdots ,a_n,[a_{n+1};a_{n+2},a_{n+3},\cdots ]\big]	\\
	&=\frac{p_n[a_{n+1};a_{n+2},a_{n+3},\cdots ]+p_{n-1}}{q_n[a_{n+1};a_{n+2},a_{n+3},\cdots ]+q_{n-1}},
\end{align*}
\begin{align} \nonumber
 a -\frac{p_n}{q_n} &=\frac {(-1)^n}{q_n^2([a_{n+1};a_{n+2},\dots]+[0;a_n,\dots ,a_1])}	\\ \label{hanclJitu1.0}
 &= \frac {(-1)^n}{q_n^2([a_{n+1};a_{n+2},\dots]+\frac {q_{n-1}}{q_n})},
\end{align}
\begin{equation} \label{hanclJitu1.01}
\frac{p_n}{q_n}=\sum_{k=1}^n \frac{(-1)^{k+1}}{q_kq_{k-1}},
\end{equation}
and
\begin{equation}\label{hanclJitu1.001}
q_n=a_nq_{n-1}+q_{n-2}>a_nq_{n-1}>\dots>\prod_{k=1}^n a_k, \ a_k>0, \  k=1,\dots,n
\end{equation}
for all $n\in\mathbb N$.  If $a=[a_0;a_1,a_2,\cdots ,a_k]$ is finite and $k\geq 1$, then we suppose that $a_k\not= 1$.  All of this can be found in the book of Schmidt \cite{schmidt} pages 7-10. 
If $x$ is algebraic number and $x_1=x,x_2,\dots,x_k$ all are its different conjugates, then the house of $x$ is the maximal modulus among the conjugates of $x$ i.e. $\house{x}=\max_{1\leq j \leq k} | x_j|$. 

\section{Main Results}\label{sec:main results}
\begin{Theorem}    \label{hanclJitu1.T1}  

Let $D$ and $M$ be positive integers, and let $\gamma\in (0,1)$. Let $\mathbb{K}$ be an algebraic field such that $\deg\mathbb K=D$.  For every $j=1,\dots ,M$ let $\{ S_{n,j}\}_{n=1}^\infty$,   $\{ a_{n,j}\}_{n=1}^\infty$, $\{ b_{n,j}\}_{n=1}^\infty$ , $\{ c_{n,j}\}_{n=1}^\infty$,  $\{ d_{n,j}\}_{n=1}^\infty$ be the sequences and  $\alpha_j=[a_{0,j};a_{1,j},\dots]$  be continued fractions such that for all  $n\in\mathbb N$, we have $S_{n,j}\in\mathbb {Z}$, $b_{n,j}, c_{n,j}, d_{n,j}\in\mathbb {K}$ are algebraic integers,
\begin{equation} \label{hanclJitu1.0001}
a_{n,j}= \frac{S_{n,j}b_{n,j}+ c_{n,j}}{d_{n,j}}\geq1,
\end{equation}
\begin{equation} \label{hanclJitu1.00001}
\house{\frac{1}{a_{n,j}}}\ge \frac{1}{2},
\end{equation}
\begin{equation}\label{hanclJitu1.3}
a_{n,1}<\max(a_{n,M}2^{(\log_2a_{n,M})^\gamma},2^{d^{n\gamma}}),
\end{equation}
and 
\begin{equation}          \label{hanclJitu1.1}
\house{b_{n,j}},\house{c_{n,j}},\house{d_{n,j}}\leq\max (2^{(\log_2 a_{n,j})^\gamma},2^{d^{n\gamma}}).
\end{equation}
 Assume that
\begin{equation}          \label{hanclJitu1.2}
\liminf_{n\to\infty} \sqrt n\Big(\frac {a_{n,j}}{a_{n,j+1}}-1\Big)>0
\end{equation}
for all $j=1,\dots ,M-1$. 
Set $d=\max (2,DM-1)$. Suppose that
\begin{equation}\label{hanclJitu1.4}
\limsup_{n\to\infty}{a_{n,j}^{\frac{1}{d^n}}}=\infty .
\end{equation}
Then the numbers $\alpha_1,\dots ,\alpha_M$ and $1$ are linearly independent over $\mathbb K$. 

This result is also true if instead of \eqref{hanclJitu1.00001},
\begin{equation}\label{laursen1.1}
	(-1)^{e_{j,\sigma}}\Re(\sigma a_{n,j})\ge \max(2^{(\log_2a_{n,1})^\gamma},2^{d^{n\gamma}})^{-1}
\end{equation}
holds for all $j=1,\ldots,M$ and all embeddings $\sigma$ of $\mathbb{K}$ into $\overline{\mathbb{Q}}$, where $e_{j,\sigma}\in\{0,1\}$ does not depend on $n$.
\end{Theorem}
\begin{Theorem}                \label{hanclJitu1.T2}
 Let $K\geq 2$ be an integer  and let $\{a_n\}_{n=1}^\infty$ be a sequence of integers greater or  equal to $K$ and  such that 
$\limsup_{n\rightarrow\infty} a_n^{\frac{1}{(K2^{\pi(K)})^n}}=\infty .$
Then the continued fractions $[0;\sqrt j a_1, \sqrt j a_2, \dots ]$, $j=1,2,\dots ,K$ and the number $1$ are linearly independent over $ \mathbb{Q}(\sqrt 1, \sqrt 2, \dots , \sqrt K) $ particularly over $\mathbb Q$.
\end{Theorem}
\begin{Corollary}  \label{hanclJitu1.C2}
Let $K$ be a positive integer  and let $\{a_n\}_{n=1}^\infty$ be a sequence of positive integers greater or equal to $K$ and such that 
$\limsup_{n\rightarrow \infty} a_n^{\frac{1}{d^n}}=\infty$
where $d=\max (2,K-1)$.
Then the continued fractions $[0;\frac{a_1}{j},\frac{a_2}{j}, \dots ]$, $j=1,2,\dots ,K$ and the number $1$ are linearly independent over $\mathbb{Q}$.
\end{Corollary}
\begin{Example} \label{hanclJitu1.Ex3}
Let $\{a_n\}_{n=1}^\infty$ be a sequence of positive integers greater or equal to 3 and such that $\limsup_{n\to\infty}a_n^{\frac{1}{2^n}}=\infty$. 
Then the continued fractions $[0;a_1,a_2,\dots]$, $[0;\frac{a_1}{2},\frac{a_3}{2},\dots]$,  $[0;\frac{a_3}{3},\frac{a_3}{3},\dots]$ and number $1$ are linearly independent over $\mathbb Q$.
\end{Example}
\begin{Corollary} \label{hanclJitu1.C3}
Let $K\geq2$ be an integer and $P(x)$ be a polynomial with integer coefficients and $\deg P=K$. Assume that all roots $\alpha_1,\alpha_2,\dots,\alpha_K$ of $P(x)$ are real, different and greater than 1. Let  $\{a_n\}_{n=1}^\infty$ be the sequence of positive integers such that 
$\limsup_{n\rightarrow \infty} a_n^{\frac{1}{(K^2-1)^n}}=\infty$.
Then the continued fractions $[0;\alpha_j a_1,\alpha_j a_2,\dots], j=1,\dots,K$ and the number $1$ are linearly independent over $\mathbb {Q}( \alpha_1,\dots,\alpha_K)$ particularly over $\mathbb Q$.
\end{Corollary}
\begin{Example} \label{hanclJitu1.Ex1}
Let $\{a_n\}_{n=1}^\infty$ be the sequence of positive integeres such that $\limsup_{n\rightarrow \infty} a_n^{\frac{1}{3^n}}=\infty$. 
Then the continued fractions 
$$\big[0;\big(4+\sqrt2\big)a_1,\big(4+\sqrt2\big)a_2,\dots\big], \  \big[0;\big(4-\sqrt2\big)a_1,\big(4-\sqrt2\big)a_2,\dots\big]$$ 
and the number $1$ are linearly independent over $\mathbb{Q}(\sqrt 2)$ particularly they are linearly independent over $\mathbb{Q}$.
\end{Example}
\begin{Example} \label{laursen.Ex}
	Let $K$ be a positive integer, let $\varphi = (\sqrt{5}+1)/2$ be the golden ratio, and let $\{a_n\}_{n=1}^\infty$ be a sequence of positive integers such that $\limsup_{n\rightarrow \infty} a_n^{\frac{1}{d^n}}=\infty$ where $d=\max\{2,2K-1\}$.
	Then the continued fractions
	\begin{align*}
		&\big[0;\varphi a_1,\varphi^3 a_2,\varphi^5 a_3,\dots\big], \  \big[0;\varphi^2 a_1,\varphi^4 a_2, \varphi^6 a_3\dots\big], \ \ldots,\\ 
		&\big[0;\varphi^K a_1, \varphi^{K+2} a_2, \varphi^{K+4} a_3,\ldots\big],
	\end{align*} 
	and the number $1$ are linearly independent over $\mathbb{Q}(\alpha)$.
\end{Example}
\begin{Corollary} \label{hanclJitu1.C4}
Let $K$ be an integer and let $\{a_n\}_{n=1}^\infty$ be a sequence of positive integers such that 
$\limsup_{n\to\infty}a_n^{\frac{1}{d^n}}=\infty$
where $d=\max(2, K-1)$. Then the continued fractions $[0;a_1(j+\sum_{k=1}^1 \frac{1}{k}),\dots,a_n(j+\sum_{k=1}^n \frac{1}{k}),\dots]$,
$j=1,\dots,K$ and number $1$ are linearly independent over $\mathbb Q$.
\end{Corollary}
\begin{Corollary} \label{hanclJitu1.C5}
Let $K$ be an integer and $\{a_n\}_{n=0}^\infty$ be a sequence of positive integers such that 
$\limsup_{n\to\infty} a_n^{\frac{1}{d^n}}=\infty$ 
where $d=\max(2,K-1)$. Then the continued fractions $[0;a_1(1+\frac{\pi(1)}{1})^j,\dots,a_n(1+\frac{\pi(n)}{n})^j,\dots]$, $j=0,1,\dots,K-1$ and number $1$ are linearly independent over $\mathbb Q$.
\end{Corollary}
\begin{Example} \label{hanclJitu1.Ex2}
Let $\{a_n\}_{n=1}^\infty$ be the sequence of positive integer such that 
$\limsup_{n\to\infty}a_n^{\frac{1}{3^n}}=\infty$. 
Then the continued fractions 
$$\big[0;a_1d(1)\sqrt2,\dots,a_nd(n)\sqrt2,\dots\big], \ \big[0;a_1(1+d(1))\sqrt2,\dots,a_n(1+d(n))\sqrt2,\dots\big],$$ 
and the number $1$ are linearly independent over $\mathbb Q(\sqrt2)$ particularly over $\mathbb Q$.
\end{Example}
\begin{Question}
Does every sequence  $\{a_n\}_{n=1}^\infty$ of positive integers such that $\limsup_{n\to\infty}a_n^{\frac{1}{2^n}}=\infty$, then the continued fractions $[0;a_1+\sqrt2,a_2+\sqrt2,\dots]$ ,and $[0;a_1+\sqrt3,a_2+\sqrt3,\dots]$ are linearly independent over $\mathbb Q$?
\end{Question}
\begin{Question}
Check if there exists a sequence $\{a_n\}_{n=1}^\infty$ of positive integers such that the continued fractions $[0;a_1,a_2,\dots]$ and $[0;a_1+1,a_2+2,\dots]$ are linearly dependent over $\mathbb Q$?
\end{Question}
\begin{Lemma}    \label{hanclJitu1.l1}
Let $\{a_n\}_{n=1}^\infty$ and ${\{b_n\}_{n=1}^\infty}$ be two sequences of real numbers greater or equal to 1 and such that
\begin{equation}          \label{hanclJitu1.10}
\liminf_{n\to\infty} \sqrt n\Big(\frac {a_n}{b_n}-1\Big)>0.
\end{equation}
Let $q_{n,a}$ and $q_{n,b}$ be denominator of $n$-th partial of the continued fraction $a=[0;a_1,a_2,\dots]$ and $b=[0;b_1,b_2,\dots]$, respectively. Then 
\begin{equation}              \label{hanclJitu1.11}
\lim_{n\to\infty}\frac{q_{n,a}}{q_{n,b}}=\infty .
\end{equation}
\end{Lemma} 
\begin{Question}
Is that possible to substitute condition (\ref{hanclJitu1.10}) by  the condition $\limsup_{n\to\infty}n(\frac{a_n}{b_n}-1)>0$?
\end{Question}
\begin{Remark}
We cannot substitute condition (\ref{hanclJitu1.10}) by the weaker condition $\limsup_{n\to\infty}n^2(\frac{a_n}{b_n}-1)>0$. For example set $a_n=n^2+1$ and $b_n=n^2$ for all $n\in\mathbb N$. Then 
\begin{align}
\frac{q_{n,a}}{q_{n,b}}= \nonumber &\frac{(n^2+1)q_{n-1,a}+q_{n-2,a}}{n^2q_{n-1,b}+q_{n-2,b}}
=\frac{n^2+1}{n^2}\frac{q_{n-1,a}}{q_{n-1,b}}\left(\frac{1+\frac{q_{n-2,a}}{(n^2+1)q_{n-2,a}}}{1+\frac{q_{n-2,b}}{n^2q_{n-2,b}}}\right)\\ \nonumber
< & \Big(1+\frac{1}{n^2}\Big)^2\frac{q_{n-1,a}}{q_{n-2,b}}
< \dots <  \prod_{n=1}^\infty \Big(1+\frac{1}{n^2}\Big)^2= const<\infty .
\end{align}
\end{Remark}

\begin{Lemma}\label{hanclJitu1.l2}
Let $z_0, z_1,z_2,\dots,z_n$ be complex numbers such that for every $k=0,1,2,\dots,n$ we have 
\begin{equation}\label{hanclJitu1.l21}
| z_k|\geq2.
\end{equation}
Then
\begin{equation}\label{hanclJitu1.l22}
    \big| [z_0; z_1,z_2,\dots,z_n] \big|\geq1.
\end{equation}
\end{Lemma}

\begin{Lemma}\label{laursen1.l1}
	Let $z_0, z_1,z_2,\dots,z_n$ be complex numbers with $\Re(z_k)>0$ for every $k=0,1,2,\dots,n$.
	Then
	\begin{equation*}
		|\Re\big([-z_0; -z_1,-z_2,\dots,-z_n]\big)| = \Re\big([z_0; z_1,z_2,\dots,z_n]\big) \ge \Re(z_0).
	\end{equation*}
\end{Lemma}
\section{Proofs}
\begin{proof}[Proof of Lemma \ref{hanclJitu1.l2} ]
 Lemma \ref{hanclJitu1.l2} can be proved by mathematical induction using the inequality $| a+b|\geq|| a|-| b||$, which holds for all complex numbers $a$ and $b$.
\end{proof}
\begin{proof}[Proof of Lemma \ref{laursen1.l1} ]
	Lemma \ref{laursen1.l1} can be proved by mathematical induction using that $\Re(a+b^{-1}) = \Re(a) + \Re(b)/|b|^2$, which holds for all non-zero complex numbers $a$ and $b$.
\end{proof}

\begin{proof}[Proof of Theorem \ref{hanclJitu1.T1}]

Let $\sigma_1,\dots ,\sigma_D$ be the set of embeddings of $\mathbb{K}$ into $\overline{\mathbb{Q}}$, where $\sigma_1$ is the identity. Then the number $\prod_{i=1}^D\sigma_i(x)$ is a rational number for every $x\in\mathbb K$. 
 
If $b_{n,j}\not=0$, then we have that $| \prod_{i=1}^D\sigma_i(b_{n,j}) |=A$ where $A$ is a positive integer. 
From this and (\ref{hanclJitu1.1}) we obtain that for every $i=1,\dots,D$
\begin{equation}\label{hanclJitu1.4a}
| \sigma_i(b_{n,j}) |=\frac{A}{| \prod_{I=1, I\not=i}^D\sigma_I(b_{n,j}) |}  \geq \frac{1}{\max(2^{(D-1)(\log_2 a_{n,j})^\gamma},2^{(D-1)d^{n\gamma}})}
\end{equation}
and similarly
\begin{equation}\label{hanclJitu1.4c}
| \sigma_i(d_{n,j}) |  \geq \frac{1}{\max(2^{(D-1)(\log_2 a_{n,j})^\gamma},2^{(D-1)d^{n\gamma}})}.
\end{equation}
Inequalities (\ref{hanclJitu1.0001}), (\ref{hanclJitu1.1}) and (\ref{hanclJitu1.4a}) yield
\begin{equation}\label{hanclJitu1.4b}
| S_{n,j}|=\Big| \frac{d_{n,j}a_{n,j}-c_{n,j}}{b_{n,j}}\Big| \leq 2a_{n,j}\max(2^{(D\log_2 a_{n,j})^\gamma},2^{Dd^{n\gamma}}).
\end{equation}

Suppose that the numbers $\alpha_1,\alpha_2,\dots,\alpha_M$ and $1$ are linearly dependent over $\mathbb K$. Then there exist ${A_1,A_2,\dots,A_M}\in\mathbb K$, not all equal to zero such that ${\sum_{j=1}^M{A_j\alpha_j}}\in\mathbb K$. Let us write
${\sum_{j=1}^M{A_j\alpha_j}=y } $ 
where $y\in\mathbb K$. Therefore
$\sum_{j=1}^M{A_j\alpha_j}-y=0.$
Let $n_0$ and $n$ be sufficiently large and such that $n\geq n_0$. Then we have 
$$\sum_{j=1}^M{A_j\alpha_j}-y=\sum_{j=1}^M {A_j\Big(\alpha_j-\frac{p_{n,j}}{q_{n,j}}\Big)}+\sum_{j=1}^M A_j \frac{p_{n,j}}{q_{n,j}}-y=0.$$
Hence 
\begin{equation} \label{hanclJitu1.5}
\sum_{j=1}^M A_j \frac{p_{n,j}}{q_{n,j}}-y=-\sum_{j=1}^M {A_j\Big(\alpha_j-\frac{p_{n,j}}{q_{n,j}}\Big)}.
\end{equation}
Now we prove that for all large $n$ we have
\begin{equation}             \label{hanclJitu1.6}
\bigg|\sum_{j=1}^MA_j\frac{p_{n,j}}{q_{n,j}}-y\bigg|>0.
\end{equation}
Without loss of generality let $M^*$ be an integer such that $A_M=\cdots=A_{M^*+1}=0$ and $A_{M^*}\neq 0$. From this and \eqref{hanclJitu1.0}, we obtain that
\begin{align*}
\sum_{j=1}^M {A_j(\alpha_j-\frac{p_{n,j}}{q_{n,j}})} &=\sum_{j=1}^{M^*} {A_j(\alpha_j-\frac{p_{n,j}}{q_{n,j}})} \\
& =\sum_{j=1}^{M^*} \frac{A_j(-1)^n}{q_{n,j}^2(a_{n+1,j}+[0;a_{n+2,j},a_{n+3,j},\dots]+[0;a_{n,j},\dots,a_{1,j}])}.
\end{align*}
This and Lemma \ref{hanclJitu1.l1} yield
\begin{align*}
&\bigg|\sum_{j=1}^{M} A_j\Big(\alpha_j-\frac{p_{n,j}}{q_{n,j}}\Big)\bigg|	\\
& \geq  \frac{|A_{M^*}|}{q_{n,M^*}^2(a_{n+1,M^*}+[0;a_{n+2,j},a_{n+3,j},\dots]+[0;a_{n,j},\dots,a_{1,j}])}-\sum_{j=1}^{M^*-1}\frac{|A_j|}{q_{n,j}^2a_{n+1,j}}\\
&\geq  \frac{|A_{M^*}|}{3q_{n,M^*}^2a_{n+1,M^*}}-\sum_{j=1}^{M^*-1}\frac{|A_j|}{q_{n,j}^2a_{n+1,j}} \\
&=  \frac{|A_{M^*}|}{3q_{n,M^*}^2a_{n+1,M^*}}\bigg(1-\sum_{j=1}^{M^*-1}\frac{\frac{3|A_j|}{|A_{M^*}|}}{(\frac{q_{n,j}}{q_{n,M^*}})^2\frac{a_{n+1,j}}{a_{n+1,M^*}}}\bigg)> 0. 
\end{align*}
This and  (\ref{hanclJitu1.5}) yield (\ref{hanclJitu1.6}). 

Using (\ref{hanclJitu1.0}) yields that
\begin{align}
\nonumber \bigg|\sum_{j=1}^{M} A_j(\alpha_j-\frac{p_{n,j}}{q_{n,j}})\bigg| &=\bigg|\sum_{j=1}^{M^*} A_j(\alpha_j-\frac{p_{n,j}}{q_{n,j}})\bigg| \\ \nonumber &=\bigg|\sum_{j=1}^{M^*} \frac{A_j(-1)^n}{q_{n,j}^2(a_{n+1,j}+[0;a_{n+2,j},a_{n+3,j},\dots]+[0;a_{n,j},\dots,a_{1,j}])}\bigg| \\ \label{hanclJitu1.7}  
& \leq\sum_{j=1}^{M^*}\frac{\max_{j=1,2,\dots,M^*}(| A_j|)}{q_{n,j}^2a_{n+1,j}}.
\end{align}
From this, Lemma \ref{hanclJitu1.l1} and (\ref{hanclJitu1.2}) we obtain that
\begin{align*}
 \bigg|\sum_{j=1}^M A_j(\alpha_j-\frac{p_{n,j}}{q_{n,j}^2})\bigg|\leq & \max_{j=1,,\dots,M}(| A_j |)\sum_{j=1}^M\frac{1}{q_{n,j}^2a_{n+1,j}} \\ 
=& \frac{\max_{j=1,2,\dots,M}(| A_j |)}{q_{n,M}^2a_{n+1,M}}\bigg(1+\sum_{j=1}^{M-1}(\frac{q_{n,M}}{q_{n,j}})^2\frac{a_{n+1,M}}{a_{n+1,j}}\bigg) \\ 
\leq & \frac{M\max_{j=1,2,\dots,M}| A_j |}{q_{n,M}^2a_{n+1,M}}
=\frac{c}{q_{n,M}^2a_{n+1,M}} 
\end{align*}
where $c=M\max_{j=1,2,\dots,M}| A_j |$ is  a constant which does not depend on $n$.
This, (\ref{hanclJitu1.5}), and (\ref{hanclJitu1.6}) yield
$$0<\bigg|\sum_{j=1}^MA_j\frac{p_{n,j}}{q_{n,j}}-y\bigg|< \frac{c}{q_{n,M}^2a_{n+1,M}}.$$
It implies that
$$\prod_{i=1}^D\bigg| \sigma_i\bigg(\sum_{j=1}^MA_j\frac{p_{n,j}}{q_{n,j}}-y\bigg)\bigg|\leq \prod_{i=2}^D\bigg| \sigma_i\bigg(\sum_{j=1}^MA_j\frac{p_{n,j}}{q_{n,j}}-y\bigg)\frac{c}{q_{n,M}^2a_{n+1,M}}. $$
Hence,
\begin{equation}\label{hanclJitu1.e8}
\frac{c\; \big|\prod_{i=2}^D\sum_{j=1}^M (\sigma_i(A_j\frac{p_{n,j}}{q_{n,j}})-\sigma_i(y))\big|}{q_{n,M}^2a_{n+1,M}}\geq
\bigg|\prod_{i=1}^D\sum_{j=1}^M \Big(\sigma_i\Big(A_j\frac{p_{n,j}}{q_{n,j}}\Big)-\sigma_i(y)\Big)\bigg|.
\end{equation}
From (\ref{hanclJitu1.6}) and by Galois theory the number on the right is a rational number 
$\frac{a}{b},(a,b)=1, a,b \in\mathbb Z^+$ such that there exist a constant $C$ which does not depend on $n$ and such that  
\begin{align*}
0<b&\leq C \prod_{i=1}^D \prod_{j=1}^M \bigg|\sigma_i\bigg(\bigg(\prod_{k=1}^n d_{k,j}\bigg)q_{n,j}\bigg)\bigg| \\
&=C\prod_{i=1}^D\prod_{j=1}^M |\sigma_i(q_{n,j})|\prod_{k=1}^n| \sigma_i(d_{k,j}) |.
\end{align*}
This and (\ref{hanclJitu1.1}) yield
\begin{equation}
\label{hanclJitu1.09} 0<b\leq C \prod_{i=1}^D\prod_{j=1}^M |\sigma_i(q_{n,j})| \prod_{k=1}^n\max (2^{(\log_2 a_{k,j})^\gamma},2^{d^{k\gamma}}).
\end{equation}
Write $R_{k,j} = \max(2^{(\log_2 a_{k,j})^\gamma}, 2^{d^{k\gamma}})$ and notice that $\house{1/a_{k,j}}\le R_{k,j}$ regardless of whether we assume \eqref{hanclJitu1.00001} or \eqref{laursen1.1}.
Combining this with inequality (\ref{hanclJitu1.0001}), we find
\begin{align*}
|\sigma_i(q_{n,j})|&=|\sigma_i(a_{n,j}q_{n-1,j}+q_{n-2,j})| \\
&\leq|\sigma_i(a_{n,j})||\sigma_i(q_{n-1,j})|+|\sigma_i(q_{n-2,j})|	\\
&\leq |\sigma_i(q_{n-2,j})|\big(	|\sigma_i(a_{n,j})||\sigma_i(a_{n-1,j})| +1	\big) + |\sigma_i(a_{n,j})||\sigma_i(q_{n-3,j})|	\\
&< (1+R_{n,j} R_{n-1,j})
|\sigma_i(a_{n,j})|\big(|\sigma_i(a_{n-1,j})||\sigma_i(q_{n-2,j})|+|\sigma_i(q_{n-3,j})|\big)	\\
&<\cdots< \prod_{k=1}^n (1+R_{k,j} R_{k-1,j})|\sigma_i(a_{k,j})| < 2^n R_{n,j}\prod_{k=1}^{n-1} R_{k,j}^2|\sigma_i(a_{k,j})|
\\&
<\frac{\prod_{k=1}^{n} R_{k,j}^2|\sigma_i(a_{k,j})|}{C^{1/DM}}
= \frac{\prod_{k=1}^{n} \max(2^{(\log_2 a_{k,j})^\gamma}, 2^{d^{k\gamma}})|\sigma_i(a_{k,j})|}{C^{1/DM}}.
\end{align*}
This and (\ref{hanclJitu1.09}) yield
\begin{align*}
0<b&\leq C \prod_{i=1}^D\prod_{j=1}^M |\sigma_i(q_{n,j})|\prod_{k=1}^n\max (2^{(\log_2 a_{k,j})^\gamma},2^{d^{k\gamma}})\\
&< \prod_{i=1}^D\prod_{j=1}^M \prod_{k=1}^n|\sigma_i(a_{k,j})| \max (2^{(\log_2 a_{k,j})^\gamma},2^{d^{k\gamma}})^3.
\end{align*}
From this, (\ref{hanclJitu1.0001}) and (\ref{hanclJitu1.4c}) we obtain that 
\begin{align*}
0<b&\leq \prod_{i=1}^D\prod_{j=1}^M\prod_{k=1}^n
\frac{|\sigma_i(S_{k,j}b_{k,j}+c_{k,j})|}{| \sigma_i(d_{k,j})|}\max(2^{(\log_2a_{k,j})^\gamma} ,2^{d^{k\gamma}})^3 \\
&\leq \prod_{i=1}^D \prod_{j=1}^M \prod_{k=1}^n\left(| S_{k,j}||\sigma_i(b_{k,j})|+|\sigma_i(c_{k,j})|\right)
\max(2^{(\log_2 a_{k,j})^\gamma},2^{d^{k\gamma}})^{D+2}.
\end{align*}
This and (\ref{hanclJitu1.1}) imply that
\begin{equation*}
0<b< \prod_{i=1}^D \prod_{j=1}^M \prod_{k=1}^n\left(| S_{k,j}|+1\right) \max(2^{(\log_2 a_{k,j})^\gamma},2^{d^{k\gamma}})^{D+3}.
\end{equation*}
This and (\ref{hanclJitu1.4b}) yield
\begin{align*}
 0<b&<\prod_{i=1}^D \prod_{j=1}^M \prod_{k=1}^n\left(2a_{k,j}\max(2^{(D\log_2 a_{k,j})^\gamma},2^{Dd^{k\gamma}})+1\right)\times \\
&\quad\ \max(2^{(\log_2 a_{k,j})^\gamma},2^{d^{k\gamma}})^{D+3} \\
&\leq \prod_{i=1}^D \prod_{j=1}^M \prod_{k=1}^n \left(a_{k,j} \max(2^{(\log_2 a_{k,j})^\gamma},2^{d^{k\gamma}})^{5D}\right) 
\end{align*}
This and (\ref{hanclJitu1.2}) imply that
\begin{align}	
	\nonumber
	0<b&< \prod_{i=1}^D \prod_{j=1}^M \prod_{k=1}^n \left(a_{k,1} \max(2^{(\log_2 a_{k,1})^\gamma},2^{d^{k\gamma}})^{5D}\right) 
	\\&\label{hanclJitu1.021}
	= \prod_{k=1}^n \left(a_{k,1}^{DM} \max(2^{(\log_2 a_{k,1})^\gamma},2^{d^{k\gamma}})^{5D^2M}\right) 
\end{align}

We now calculate
\begin{align}
	\nonumber |\sigma_i(q_{n,j})|=&|\sigma_i(a_{n,j})\sigma_i(q_{n-1,j})+\sigma_i(q_{n-2,j}) | \\ \nonumber
	= & |\sigma_i(q_{n-1,j})|\bigg|\sigma_i(a_{n,j})+\frac{\sigma_i(q_{n-2,j})}{\sigma_i(q_{n-1,j})} \bigg| \\ \nonumber
	=& |\sigma_i(q_{n-1,j})||\sigma_i(a_{n,j})+[0;\sigma_i(a_{n-1,j}),\sigma_i(a_{n-2,j}),\dots,\sigma_i(a_{1,j})]| \\ \label{eq:sigma i qnj}
	=& |\sigma_i(q_{n-1,j})||[\sigma_i(a_{n,j});\sigma_i(a_{n-1,j}),\sigma_i(a_{n-2,j}),\dots,\sigma_i(a_{1,j})]|.
\end{align}
If \eqref{hanclJitu1.00001} is satisfied, we then apply Lemma \ref{hanclJitu1.l2} to \eqref{eq:sigma i qnj} and find
\begin{equation*}
	|\sigma_i(q_{n,j})|\geq  |\sigma_i(q_{n-1,j})|\geq |\sigma_i(q_{n-2,j})|\geq\dots\geq|\sigma_i(q_{0,j})|=1,
\end{equation*}
while \eqref{laursen1.1} would allow us to apply Lemma \ref{laursen1.l1} to \eqref{eq:sigma i qnj} and obtain
\begin{align*}
	|\sigma_i(q_{n,j})| &\ge |\sigma_i(q_{n-1,j})| |\Re([\sigma_i(a_{n,j});\sigma_i(a_{n-1,j}),\sigma_i(a_{n-2,j}),\dots,\sigma_i(a_{1,j})])|
	\\&
	\ge |\sigma_i(q_{n-1,j})| |\Re \sigma_i(a_{n,j})|
	\ge \frac{|\sigma_i(q_{n-1,j})|}{\max(2^{(\log_2 a_{k,1})^\gamma},2^{d^{k\gamma}})}
	\\&
	\ge\cdots\ge \prod_{k=1}^{n} \max(2^{(\log_2 a_{k,j})^\gamma},2^{d^{k\gamma}})^{-1}
\end{align*}
Whether \eqref{hanclJitu1.00001} or \eqref{laursen1.1} is true, this and (\ref{hanclJitu1.01}) then yield
\begin{align}
	\nonumber & \left| \prod_{i=2}^D\left(\sum_{j=1}^M(\sigma_i(A_j\frac{p_{n,j}}{q_{n,j}})-\sigma_i(y))\right)\right| \\ \nonumber
	&=\left| \prod_{i=2}^D\left(\sum_{j=1}^M(\sigma_i(A_j\sum_{k=1}^n\frac{(-1)^{k+1}}{q_{k,j}q_{k-1,j}}-\sigma_i(y)))\right)\right| \\ \nonumber
	&\leq \prod_{i=2}^D\left(\sum_{j=1}^M|\sigma_i(A_j)|\sum_{k=1}^n\frac{1}{|\sigma_i(q_{k,j})||\sigma_i(q_{k-1,j})|}\right)
	\\&\nonumber
	\le \prod_{i=2}^D\left(\sum_{j=1}^M|\sigma_i(A_j)|\sum_{k=1}^n \max(2^{(\log_2 a_{k,j})^\gamma},2^{d^{k\gamma}}) \prod_{l=1}^{k-1}\max(2^{(\log_2 a_{l,j})^\gamma},2^{d^{l\gamma}})^2\right)
	\\&\label{hanclJitu1.011}
	\leq \prod_{k=1}^{n}\max(2^{(\log_2 a_{k,1})^\gamma},2^{d^{k\gamma}})^{2D},
\end{align}
for all sufficiently large $n$.

From (\ref{hanclJitu1.e8}),(\ref{hanclJitu1.021}), and (\ref{hanclJitu1.011}), we obtain that for all sufficiently large $n$,

\begin{align*}
	q_{n,M}^2a_{n+1,M} &< \frac{1}{c}\prod_{k=1}^{n} \left(a_{n,1}^{DM} \max(2^{(\log_2 a_{k,1})^\gamma},2^{d^{k\gamma}})^{7D^2 M}\right)
	\\&
	\le \Bigg(\prod_{k=1}^{n} a_{n,1}^{DM}\Bigg) \Bigg(\prod_{k=1}^{n} 2^{7D^2 M(\log_2 a_{k,1})^\gamma} \Bigg) \Bigg(\prod_{k=1}^{n} 2^{7D^2 Md^{k\gamma}} \Bigg)
	\\&
	= 2^{7 D^2M \frac{d^{\gamma (n+1)}}{d^\gamma-1}} \Bigg(\prod_{k=1}^{n} a_{n,1}^{DM}\Bigg) \Bigg(\prod_{k=1}^{n} 2^{7D^2 M(\log_2 a_{k,1})^\gamma} \Bigg)
	\\&
	\le 2^{7 D^2M \frac{d^{\gamma (n+1)}}{d^\gamma-1}} \Bigg(\prod_{k=1}^{n} a_{n,1}^{DM}\Bigg) \Bigg(\prod_{k=1}^{n} 2^{7D^2 M(\log_2 a_{k,1})^\gamma} \Bigg)
\end{align*}
From this and (\ref{hanclJitu1.3}) we obtain that
\begin{align*}
	\nonumber q_{n,M}^2a_{n+1,M}&< 2^{7 D^2M \frac{d^{\gamma (n+1)}}{d^\gamma-1}} \left(\prod_{k=1}^n\max(a_{k,M}2^{(\log_2a_{k,M})^\gamma},2^{d^{k\gamma}})^{DM}\right)\times \\  
	&\quad \left(\prod_{k=1}^n 2^{7D^2M(\log_2 \max(a_{k,M}2^{(\log_2a_{k,M})^\gamma},2^{d^{k\gamma}}))^\gamma}\right).
\end{align*}
Set $b_k=\max\{a_{k,M},2^{d^{k\gamma}}\}$.
From this and (\ref{hanclJitu1.021}) we obtain for sufficiently large $n$ that 
\begin{align}
	\nonumber q_{n,M}^2a_{n+1,M}&< 2^{7 D^2M \frac{d^{\gamma (n+1)}}{d^\gamma-1}} \left( \prod_{k=1}^n b_k 2^{(\log_2 b_k)^\gamma} \right)^{DM} \left( \prod_{k=1}^n2^{8D^2M(\log_2 b_k)^\gamma} \right) \\ \label{hanclJitu1.010}
	&\leq 2^{7 D^2M \frac{d^{\gamma (n+1)}}{d^\gamma-1}} \left(\prod_{k=1}^n b_k\right)^{DM}\left(\prod_{k=1}^n 2^{8D^2M(\log_2 b_k)^\gamma}\right).
\end{align}

Inequality (\ref{hanclJitu1.001}) implies that 
\begin{align*}%\label{hanclJitu1.012}
q_{n,M}&>\prod_{k=1}^n a_{k,M}\geq\prod_{k=1}^nb_k\prod_{k=1}^n 2^{-d^{k\gamma}}
\ge 2^{-\frac{d^{\gamma (n+1)}}{d^\gamma-1}}\prod_{k=1}^nb_k
.
\end{align*}
This and inequality \eqref{hanclJitu1.010} imply

\begin{equation}\label{hanclJitu1.013}
	a_{n+1,M} < 2^{9 D^2M \frac{d^{\gamma (n+1)}}{d^\gamma-1}} \left(\prod_{k=1}^n b_k\right)^{d-1}\left(\prod_{k=1}^n 2^{8D^2M(\log_2 b_k)^\gamma}\right)
\end{equation}

Now the proof falls into two cases:

\textbf{Case 1}

Assume that there is  $\delta>0$ such that  
\begin{equation}\label{hanclJitu1.014}
  \limsup_{n\rightarrow\infty}
  b_n^{\frac{1}{(d+\delta)^n}} = \infty.
\end{equation}
From this and Borel's theorem we obtain that there exist infinitely many  $N$ such that
$$b_{N+1}^{\frac{1}{(d+\delta)^{N+1}}} > \Big(1+\frac{1}{N^2}\Big)\max_{k=1,\ldots,N}
  b_k^{\frac{1}{(d+\delta)^k}}.
  $$
Therefore  
 \begin{equation} \label{eq:largehJ}
 b_{N+1}>\left(1+\frac{1}{N^2}\right)^{(d+\delta)^{N+1}}\left(\max_{k=1,\ldots,N}
  b_k^{\frac{1}{(d+\delta)^k}}\right)^{(d+\delta)^{N+1}}>2^{d^{N+1}}
 \end{equation}
and 
\begin{align}
\nonumber b_{N+1}&>\left(1+\frac{1}{N^2}\right)^{(d+\delta)^{N+1}} \left(\max_{k=1,\dots,N} b_k^{\frac{1}{(d+\delta)^k}}
  \right)^{(d+\delta)^{N+1}} \\ \nonumber
&> \left(1+\frac{1}{N^2}\right)^{(d+\delta)^{N+1}}\left(\max_{k=1,\dots,N} b_k^{\frac{1}{(d+\delta)^k}}
  \right)^{(d+\delta-1)((d+\delta)^N+(d+\delta)^{N-1}+\dots+1)} \\ \label{hanclJitu1.015}
 &>2^{\frac{1}{N^3}(d+\delta)^{N+1}} \left (\prod_{k=1}^N b_k \right)^{d+\delta-1}.
\end{align}
hold for infinitely many $N$. From inequatity  (\ref{eq:largehJ}) and the fact that $b_{n+1}=\max(a_{N+1,M},2^{d\gamma(N+1)})$ we obtain that $b_{N+1}=a_{N+1,M}.$
This and (\ref{hanclJitu1.015}) imply that 
$$a_{N+1,M}>\left(\prod_{k=1}^N b_k\right)^{d-1}\left(\prod_{k=1}^N b_k\right)^\delta 2^{\frac{1}{N^3}(d+\delta)^{N+1}}$$
a contradiction with (\ref{hanclJitu1.013}). 

\textbf{Case 2}

Suppose that there is no $\delta>0$ such that (\ref{hanclJitu1.014}) holds. Then for every $\delta>0$, there exist $n_1$ such that for every $n>n_1$, we have 
\begin{equation}\label{hanclJitu1.016}
b_n<2^{(d+\delta/2)^n}. 
\end{equation}
This implies that for all sufficiently large $n$,
\begin{align} 
\nonumber \prod_{k=1}^n 2^{8D^2M(\log_2 b_k)^\gamma}&\le \prod_{k=1}^n 2^{8D^2M (\log_2{2^{(d+\delta)^k}})^\gamma}
=\prod_{k=1}^n 2^{8D^2M(d+\delta)^{k\gamma}}\\ \label{hanclJitu1.017}
&=2^{8D^2M\frac{(d+\delta)^{\gamma (n+1)}-(d+\delta)^{\gamma}}{(d+\delta)^\gamma-1}}
\leq 2^{8D^2 M\frac{(d+\delta)^{\gamma (n+1)}}{(d+\delta)^\gamma-1}}.
\end{align}
Borel's theorem and (\ref{hanclJitu1.4}) yield that for infinitely many $N$,
$$b_{N+1}^{\frac{1}{d^{N+1}}}>\left(1+\frac{1}{N^2}\right) \max_{k=1,2,\dots,N}b_k^{\frac{1}{d^k}}$$
holds. It implies that
\begin{align}
\nonumber b_{N+1}>&\left(1+\frac{1}{N^2}\right)^{d^{N+1}}\left( \max_{k=1,2,\dots,N}b_k^{\frac{1}{d^k}}\right)^{d^{N+1}}\\ \nonumber
>& \left(1+\frac{1}{N^2}\right)^{d^{N+1}}\left(\max_{k=1,2,\dots,N}b_k^{\frac{1}{d^k}}\right)^{(d-1)(d^N+d^{N-1}+\dots+1)} \\ \nonumber
>& \left(1+\frac{1}{N^2}\right)^{d^{N+1}}\left(\prod_{k=1}^N b_k\right)^{d-1}.
\end{align}
Therefore
$$b_{N+1}>2^{\frac{1}{N^3}d^{N+1}}\left(\prod_{k=1}^N b_k\right)^{d-1}.$$
From this and the fact that $b_{N+1}=\max(a_{N+1,M},2^{d^{(N+1)\gamma}})$
we obtain that  $b_{N+1}=a_{N+1,M}$.
Hence
$$a_{N+1,M}>2^{\frac{1}{N^3}d^{N+1}}\left(\prod_{k=1}^N b_k\right)^{d-1}.$$
This and (\ref{hanclJitu1.017}) yield
\begin{align}
\nonumber a_{N+1,M}>& 2^{\frac{1}{N^3}d^{N+1}}\left(\prod_{k=1}^N b_k \right)^{d-1} \left(\prod_{k=1}^N 2^{7D^2M(\log_2 b_k)^\gamma}\right)\left(\prod_{k=1}^N 2^{7D^2M(\log_2 b_k)^\gamma}\right)^{-1}\\ \nonumber
\geq & 2^{\frac{1}{N^3}d^{N+1}-8D^2M\frac{(d+\delta)^{\gamma(N+1)}}{(d+\delta)^\gamma-1}}\left(\prod_{k=1}^N b_k \right)^{d-1}
\left(\prod_{k=1}^N 2^{8D^2M(\log_2 b_k)^\gamma}\right),
\end{align}
which contradicts (\ref{hanclJitu1.013}) for a sufficiently small choice of $\delta$.
\end{proof}

\begin{proof}[Proof of Theorem \ref{hanclJitu1.T2}]

We follow the proof of Theorem \ref{hanclJitu1.T1} and the exception will be only the lower estimation of partial denominaters for the continued fractions $ \alpha_2=[0;-\sqrt 2 a_1, -\sqrt2 a_2, \dots ]$ and $ \alpha_3=[0;-\sqrt 3 a_1, -\sqrt3 a_2, \dots ]$. For $\alpha_2$  we have 
\begin{align}
\nonumber q_{n+1,\alpha_2}=& -\sqrt2 a_{n+1,\alpha_2}q_{n,\alpha_2}+q_{n-1,\alpha_2} \\ \nonumber
=&-\sqrt2 a_{n+1,\alpha_2}q_{n,\alpha_2}\bigg(1+\frac{q_{n-1,\alpha_2}}{-\sqrt2 a_{n+1,\alpha_2}q_{n,\alpha_2}}\bigg) \\ \nonumber 
=&-\sqrt2 a_{n+1,\alpha_2}q_{n,\alpha_2}\bigg(1+\frac{1}{-\sqrt2 a_{n+1,\alpha_2}}\big[0;-\sqrt 2 a_n, -\sqrt2 a_{n-1}, \dots,-\sqrt2 a_1 \big]\bigg).
\end{align}
Hence
\begin{align}
\nonumber |q_{n+1,\alpha_2}|=& \sqrt2 a_{n+1,\alpha_2}|q_{n,\alpha_2}|\bigg(1+\frac{1}{\sqrt2 a_{n+1,\alpha_2}}[0;\sqrt 2 a_n, \sqrt2 a_{n-1}, \dots,\sqrt2 a_1 ]\bigg) \\ \nonumber 
\geq & |q_{n,\alpha_2}|\geq \dots\geq |q_{1,\alpha_2}|=1
\end{align}
and (\ref{hanclJitu1.011}) follows. Similarly for $\alpha_3$. 
\end{proof}

\begin{proof}[Proof of Corollary \ref{hanclJitu1.C2}]
Corollary \ref{hanclJitu1.C2} is an immediate consequence of Theorem \ref{hanclJitu1.T1}  when we set $D=1$ and $M=K$.
\end{proof}

\begin{proof}[Proof of Example \ref{hanclJitu1.Ex3}]
Example \ref{hanclJitu1.Ex3} is the immediate consequence of Corollary \ref{hanclJitu1.C2} if we set $K=3$.
\end{proof}

\begin{proof}[Proof of Corollary \ref{hanclJitu1.C3}]
Corollary \ref{hanclJitu1.C3} is the immediate consequence of Theorem \ref{hanclJitu1.T1}  if we set $D=M=K$.
\end{proof}

\begin{proof}[Proof of Example \ref{hanclJitu1.Ex1}]
Example \ref{hanclJitu1.Ex1} is the immediate consequence of Corollary \ref{hanclJitu1.C3} when we set $K=2$ and $P(x)=x^2-8x+14$.
\end{proof}

\begin{proof}[Proof of Example \ref{laursen.Ex}]
	This is an immediate consequence of Theorem \ref{hanclJitu1.01} if we set $D=\deg \alpha$, since $\varphi^{2n}a_{n+j}$ clearly satisfies \eqref{laursen1.1}.
\end{proof}

\begin{proof}[Proof of Corollary \ref{hanclJitu1.C4}]
This is the immediate consequence of Theorem \ref{hanclJitu1.T1}  if we set $D=1$, $K=M$ and the fact that $c=\lim_{n\to\infty}(-\log n+ \sum_{j=1}^n\frac{1}{j})$ is the Euler–Mascheroni constant.
\end{proof}

\begin{proof}[Proof of Corollary \ref{hanclJitu1.C5}]
This is the immediate consequence of Theorem \ref{hanclJitu1.T1}  if we set $D=1$, $K=M$ and the fact that $\lim_{n\to\infty}\frac{\pi(n)}{\frac{n}{\log n}}=1$.
\end{proof}

\begin{proof}[Proof of Example \ref{hanclJitu1.Ex2}]
This is the immediate consequence of Theorem \ref{hanclJitu1.T1}  if we set $D=M=2$ and the well-known facts that $\liminf_{n\to\infty} d(n)=2$ and $\limsup_{n\to\infty}\frac{\log d(n)\log \log n}{\log n}=\log 2$.
\end{proof}

\begin{proof}[Proof of Lemma \ref{hanclJitu1.l1}]

From (\ref{hanclJitu1.10}) we obtain that there exists positive real number $\varepsilon$ and positive integer $n_0$ such that for every positive integer $n\geq n_0-4$ we have
\begin{equation}          \label{hanclJitu1.12}
a_n\geq\bigg(1+\frac {\varepsilon}{\sqrt {n}}\bigg)b_n,
\end{equation}
\begin{equation}          \label{hanclJitu1.13}
\frac{\varepsilon^2-\frac {\varepsilon}{\sqrt{n}+\sqrt{n+1}}}{\sqrt{n+1}(\sqrt {n}+\varepsilon)}-\frac{2}{(n+1)\ln(n+1)}>0.
\end{equation}
and the function
$$f(x)=\frac{\varepsilon^2-\frac {\varepsilon}{\sqrt{x}+\sqrt{x+1}}}{\sqrt{x+1}(\sqrt {x}+\varepsilon)}$$
is decreasing for $x>n_0$.

Set $c=[a_0,a_1,a_2,\dots,a_{n_0-5},(1+\frac {\varepsilon}{\sqrt {n_0-4}})b_{n_0-4},(1+\frac {\varepsilon}{\sqrt {n_0-3}})b_{n_0-3},\dots]$ and $q_{n,c}$ the denominator of its $n$-th partial. 
Then for all positive integers $n$ we have 
\begin{equation}          \label{hanclJitu1.14}
\frac{q_{n,a}}{q_{n,b}}\geq\frac{q_{n,c}}{q_{n,b}}.
\end{equation}
To prove Lemma \ref{hanclJitu1.l1} we prove that for every large $n$ we have 
\begin{equation}          \label{hanclJitu1.15}
\frac{q_{n+1,c}}{q_{n+1,b}}\geq \left(1+\frac{1}{(n+1)\ln{(n+1)}}\right)\frac{q_{n,c}}{q_{n,b}}.
\end{equation}
Then this, (\ref{hanclJitu1.14}) and the fact that 
$\prod_{j=n_0+1}^\infty (1+\frac{1}{j\ln j})=\infty$ imply that 
\begin{align*}
\lim_{n\to\infty}\frac{q_{n,a}}{q_{n,b}}&\geq \lim_{n\to\infty}\frac{q_{n,c}}{q_{n,b}}\geq
\lim_{n\to\infty} \Big(1+\frac{1}{n\ln{n}}\Big)\frac{q_{n-1,c}}{q_{n-1,b}} \\
& \geq \dots \geq \lim_{n\to\infty} \left(\prod_{j=n_0+1}^{n+1} (1+\frac{1}{j\ln j})\right) \frac{q_{n_0,c}}{q_{n_0,b}} \\
&=\frac{q_{n_0,c}}{q_{n_0,b}} \prod_{j=n_0+1}^\infty \Big(1+\frac{1}{j\ln j}\Big) =\infty 
\end{align*}
and  (\ref{hanclJitu1.11}) follows. 

To prove  (\ref{hanclJitu1.15}) let us set  $x_n=1+\frac {\varepsilon}{\sqrt {n}}$ and 
$y_n=1+\frac{1}{n\ln{n}}$. Then we have 
\begin{align}
\nonumber \frac{q_{n+1,c}}{q_{n+1,b}}=& \frac{x_{n+1}b_{n+1}q_{n,c}+q_{n-1,c}}{b_{n+1}q_{n,b}+q_{n-1,b}} \\ \nonumber
=& \frac{(x_{n+1}-y_{n+1}+y_{n+1})b_{n+1}q_{n,c}+q_{n-1,c}}{b_{n+1}q_{n,b}+q_{n-1,b}} \\ \nonumber
=& \frac{y_{n+1}b_{n+1}q_{n,c}\left(\frac{x_{n+1}-y_{n+1}}{y_{n+1}}+1+\frac{1}{b_{n+1}y_{n+1}\frac{q_{n,c}}{q_{n-1,c}}}\right)}
{b_{n+1}q_{n,b}\left(1+\frac{1}{b_{n+1}\frac{q_{n,b}}{q_{n-1,b}}}\right)}.
\end{align}
This yields that it is enough to prove that
$$\frac{x_{n+1}-y_{n+1}}{y_{n+1}}+1+\frac{1}{b_{n+1}y_{n+1}\frac{q_{n,c}}{q_{n-1,c}}}\geq
1+\frac  1{b_{n+1}\frac{q_{n,b}}{q_{n-1,b}}}
$$
and this is equivalent to
\begin{equation}            \label{hanclJitu1.16}
A= \frac{x_{n+1}-y_{n+1}}{y_{n+1}} + \frac{1}{b_{n+1}}
\left(\frac 1{y_{n+1}(x_nb_n+\frac{q_{n-2,c}}{q_{n-1,c}})}-\frac{1}{b_n+\frac{q_{n-2,b}}{q_{n-1,b}}} \right)\geq 0.
\end{equation}
Now we have
$$ \frac{1}{y_{n+1}({x_n}{b_n}+\frac{q_{n-2,c}}{q_{n-1,c}})}-
\frac{1}{b_n+\frac{q_{n-2,b}}{q_{n-1,b}}} =
\frac{b_n+{\frac{q_{n-2,b}}{q_{n-1,b}}-y_{n+1}x_nb_n-\frac{q_{n-2,c}}{q_{n-1,c}}{y_{n+1}}}} 
{({y_{n+1}x_nb_n+\frac{q_{n-2,c}}{q_{n-1,c}}{y_{n+1}})(b_n+\frac{q_{n-2,b}}{q_{n-1,b}}})}
$$
\begin{equation}            \label{hanclJitu1.17}
=\frac{1-y_{n+1}x_n+\frac{1}{b_n}(\frac{q_{n-2,b}}{q_{n-1,b}}-\frac{q_{n-2,c}}{q_{n-1,c}}y_{n+1})}
{(y_{n+1}x_n+\frac 1{b_n}\frac{q_{n-2,c}}{q_n-1,c}y_{n+1})(b_n+\frac{q_{n-2,b}}{q_{n-1,b}})}.
\end{equation}
This, the fact that $1-y_{n+1}x_n<0$ and $b_{n+1}\geq 1$ imply that 
$$\frac{1}{y_{n+1}({x_n}{b_n}+\frac{q_{n-2,c}}{q_{n-1,c}})}-\frac{1}{b_n+\frac{q_{n-2,b}}{q_{n-1,b}}} =
\frac{1-y_{n+1}x_n+\frac{1}{b_n}(\frac{q_{n-2,b}}{q_{n-1,b}}-\frac{q_{n-2,c}}{q_{n-1,c}}y_{n+1})}
{(y_{n+1}x_n+\frac 1{b_n}\frac{q_{n-2,c}}{q_n-1,c}y_{n+1})(b_n+\frac{q_{n-2,b}}{q_{n-1,b}})}
$$
\begin{equation}            \label{hanclJitu1.18}
\geq \frac {1-y_{n+1}x_n}{y_{n+1}x_n}+
\frac{\frac{1}{b_n}(\frac{q_{n-2,b}}{q_{n-1,b}}-\frac{q_{n-2,c}}{q_{n-1,c}}y_{n+1})}
{(y_{n+1}x_n+\frac 1{b_n}\frac{q_{n-2,c}}{q_n-1,c}y_{n+1})(b_n+\frac{q_{n-2,b}}{q_{n-1,b}})}.
\end{equation}
Set $E_{n}(y_{n+1})=\frac 1{b_n}(\frac{q_{n-2,b}}{q_{n-1,b}}-\frac{q_{n-2,c}}{q_{n-1,c}}y_{n+1})$, 
$D=(y_{n+1}x_n+\frac 1{b_n}\frac{q_{n-2,c}}{q_{n-1,c}}y_{n+1})(b_n+\frac{q_{n-2,b}}{q_{n-1,b}})$ and $F=\frac{x_{n+1}-y_{n+1}}{y_{n+1}}+\frac{1-y_{n+1}x_n}{y_{n+1}x_n}$.
This,  (\ref{hanclJitu1.17}) and  (\ref{hanclJitu1.18}) yield that 
to prove (\ref{hanclJitu1.16}), it is enough to prove that 
\begin{equation}            \label{hanclJitu1.19}
 F+{\frac{1}{b_{n+1}}\frac{E_{n}(y_{n+1})}{D}}\geq 0.
\end{equation} 
Set $E_{n-1}^*(y_{n+1})=\frac 1{b_{n-1}}(\frac{q_{n-3,c}}{q_{n-2,c}}-
\frac{q_{n-3,b}}{q_{n-2,}b}y_{n+1})  \text  { and }\\
D_{n-1}=(x_{n-1}+\frac 1{b_{n-1}}\frac{q_{n-3,c}}{q_{n-2,c}})(b_{n-1}+\frac{q_{n-3,b}}{q_{n-2,b}})$. Now we have 
\begin{align}
\nonumber E_{n}(y_{n+1})=& \frac 1{b_n}(\frac{q_{n-2,b}}{q_{n-1,b}}-\frac{q_{n-2,c}}{q_{n-1,c}}y_{n+1}) \\ \nonumber
=& \frac{1}{b_{n}} \left(\frac 1{(b_{n-1}+\frac{q_{n-3,b}}{q_{n-2,b}})}-\frac{y_{n+1}}{b_{n-1}x_{n-1}+\frac{q_{n-3,c}}{q_{n-2,c}}}\right)
\end{align}
\begin{equation}            \label{hanclJitu1.20}
=\frac 1{b_n}\frac{x_{n-1}-y_{n+1}+E_{n-1}^*(y_{n+1})}{D_{n-1}}
\end{equation} 
and 
\begin{align}
\nonumber E_{n-1}^*(y_{n+1})=& \frac 1{b_{n-1}}\left(\frac{q_{n-3,c}}{q_{n-2,c}}-
\frac{q_{n-3,b}}{q_{n-2,b}}y_{n+1}\right)\\ \nonumber
=& \frac{1}{b_{n-1}} \left(\frac 1{(b_{n-2}x_{n-2}+\frac{q_{n-4,c}}{q_{n-3,c}})}-\frac{y_{n+1}}{b_{n-2}+\frac{q_{n-4,b}}{q_{n-3,b}}}\right) \\ \nonumber
=& \frac 1{b_{n-1}}\frac{1-x_{n-2}y_{n+1}+E_{n-2}(y_{n+1})}{D_{n-2}}.
\end{align}
From this and the fact that $1-x_{n-2}y_{n+1}<0$ we obtain that 
\begin{align*}
 E_{n-1}^*(y_{n+1})=& \frac 1{b_{n-1}}\frac{1-x_{n-2}y_{n+1}+E_{n-2}(y_{n+1})}{D_{n-2}} \\
>& \frac{1-x_{n-2}y_{n+1}}{x_{n-2}}+\frac 1{b_{n-1}}\frac{E_{n-2}(y_{n+1})}{D_{n-2}}.
\end{align*}
This and (\ref{hanclJitu1.20}) imply that
\begin{align}
\nonumber E_n(y_{n+1})&=\frac 1{b_n}\frac{x_{n-1}-y_{n+1}+E_{n-1}^*(y_{n+1})}{D_{n-1}} \\  \label{hanclJitu1.21}           
&>\frac 1{b_n}\frac{x_{n-1}-y_{n+1}+\frac{1-x_{n-2}y_{n+1}}{x_{n-2}}+\frac 1{b_{n-1}}\frac{E_{n-2}(y_{n+1})}{D_{n-2}}}{D_{n-1}}.
\end{align}  

For every $j=n_0,\dots ,n$ we have
\begin{align*}
&x_{j+1}-y_{n+1}+\frac{1-x_{j}y_{n+1}}{x_{j}}\\
&=1+\frac {\varepsilon}{\sqrt {(j+1)}}-1-\frac{1}{(n+1)\ln (n+1)}+\frac{1-(1+\frac {\varepsilon}{\sqrt {j}})(1+\frac{1}{(n+1)\ln (n+1)})}{1+\frac {\varepsilon}{\sqrt {j}}}\\
&=\frac {\varepsilon}{\sqrt {(j+1)}}-\frac{1}{(n+1)\ln(n+1)}+\frac{1}{1+\frac {\varepsilon}{\sqrt {j}}}-1-\frac{1}{(n+1)\ln(n+1)} \\
&=\frac {\varepsilon}{\sqrt {(j+1)}}-\frac{2}{(n+1)\ln(n+1)}-\frac{\varepsilon}{\sqrt {j}+\varepsilon} 
=\frac{\varepsilon^2-\frac {\varepsilon}{\sqrt{j}+\sqrt{j+1}}}{\sqrt{(j+1)}(\sqrt {j}+\varepsilon)}-\frac{2}{(n+1)\ln(n+1)}.
\end{align*}
This and the fact that the function $f(x)=\frac{\varepsilon^2-\frac {\varepsilon}{\sqrt{x}+\sqrt{x+1}}}{\sqrt{x+1}(\sqrt {x}+\varepsilon)}$ is decreasing for $x>n_0$ we obtain that 
\begin{align}
\nonumber x_{j+1}-y_{n+1}+\frac{1-x_{j}y_{n+1}}{x_{j}}&= \frac{\varepsilon^2-\frac {\varepsilon}{\sqrt{j}+\sqrt{j+1}}}{\sqrt{(j+1)}(\sqrt {j}+\varepsilon)}-\frac{2}{(n+1)\ln(n+1)} \\ \label{hanclJitu1.22}
&\geq \frac{\varepsilon^2-\frac {\varepsilon}{\sqrt{n}+\sqrt{n+1}}}{\sqrt{(n+1)}(\sqrt {n}+\varepsilon)}-\frac{2}{(n+1)\ln(n+1)}.
\end{align}
From this and  (\ref{hanclJitu1.13}) we obtain that
$x_{j+1}-y_{n+1}+\frac{1-x_{j}y_{n+1}}{x_{j}}>0.$ Hence  
$x_{n-1}-y_{n+1}+\frac{1-x_{n-2}y_{n+1}}{x_{n-2}}>0.$
This and (\ref{hanclJitu1.21}) yield that
\begin{align}
\nonumber E_{n}(y_{n+1})> &
\frac 1{b_n}\frac{x_{n-1}-y_{n+1}+\frac{1-x_{n-2}y_{n+1}}{x_{n-2}}+\frac 1{b_{n-1}}\frac{E_{n-2}(y_{n+1})}{D_{n-2}}}{D_{n-1}} \\ \nonumber
 >& \frac 1{b_n}\frac{\frac 1{b_{n-1}}\frac{E_{n-2}(y_{n+1})}{D_{n-2}}}{D_{n-1}}.
\end{align}
 If we repeat this procedure then we obtain 
\begin{align}
\nonumber E_{n}(y_{n+1})>& 
\frac 1{b_n}\frac{\frac 1{b_{n-1}}(\frac{E_{n-2}(y_{n+1})}{D_{n-2}})}{D_{n-1}}>\dots >
\frac {E_{n-2[\frac n2]+2n_0-2}(y_{n+1})}{\prod_{j=n}^{n-2[\frac n2]+2n_0} b_jD_{j-1}}\\ \nonumber 
>& \frac {-| E_{n-2[\frac n2]+2n_0-2}(y_{n+1})|}{\prod_{j=n}^{n-2[\frac n2]+2n_0} x_{j-1}}=\frac {-| E_{n-2[\frac n2]+2n_0-2}(y_{n+1})|}{\prod_{j=n}^{n-2[\frac n2]+2n_0} (1+\frac \varepsilon{\sqrt{j-1}})}\\ \nonumber
=& \frac {-| E_{n-2[\frac n2]+2n_0-2}(y_{n+1})|}{e^{\sum_{j=n}^{n-2[\frac n2]+2n_0}\log (1+\frac \varepsilon{\sqrt{j-1}})}}
> \frac {-| E_{n-2[\frac n2]+2n_0-2}(y_{n+1})|}{e^{\sqrt2\varepsilon(\sqrt {n-2}-\sqrt{2n_0+2})}}.
\end{align}
This implies that to prove (\ref{hanclJitu1.19}) it is enough to prove that 
\begin{equation}            \label{hanclJitu1.23}
 F-\frac {| E_{n-2[\frac n2]+2n_0-2}(y_{n+1})|}{e^{\sqrt2\varepsilon(\sqrt {n-2}-\sqrt{2n_0+2})}}\geq 0.
\end{equation}
From    (\ref{hanclJitu1.22}) we obtain that 
\begin{align*}
 F& =  \frac{1+\frac {\varepsilon}{\sqrt {(n+1)}}-1-\frac{1}{(n+1)\ln (n+1)}}{y_{n+1}}+\frac{1-(1+\frac {\varepsilon}{\sqrt {n}})(1+\frac{1}{(n+1)\ln (n+1)})}{y_{n+1}x_n}\\
&= \frac{1}{y_{n+1}}(\frac{\varepsilon^2-\frac {\varepsilon}{\sqrt{n}+\sqrt{n+1}}}{\sqrt{(n+1)}(\sqrt {n}+\varepsilon)}-\frac{2}{(n+1)\ln(n+1)}).
\end{align*}
This implies that inequality (\ref{hanclJitu1.23}) has the form
$$\frac{1}{y_{n+1}}(\frac{\varepsilon^2-\frac {\varepsilon}{\sqrt{n}+\sqrt{n+1}}}{\sqrt{(n+1)}(\sqrt {n}+\varepsilon)}-\frac{2}{(n+1)\ln(n+1)})
-\frac {| E_{n-2[\frac n2]+2n_0-2}(y_{n+1})|}{e^{\sqrt2\varepsilon(\sqrt {n-2}-\sqrt{2n_0+2})}}\geq 0
$$
which holds for all sufficiently large $n$. The proof of Lemma \ref{hanclJitu1.l1} is complete.
\end{proof}

\textbf{Acknowledgements.} 
M.~L.~Laursen is supported by the Independent Research Fund Denmark (Grant ref. 1026-00081B).
J.B Leta is supported by grant SGS01/PřF/2024.

J. Han\v cl\ and J. Leta\\
Department of Mathematics, Faculty of Science, University of Ostrava, 30.~dubna~22, 701~03 Ostrava~1, Czech Republic.\\
e-mail: hancl@osu.cz, jiiberhanu@gmail.com\\

M. Laursen\\
Department of Mathematics, Aarhus University, Ny Munkegade 118, 8000 Aarhus C, Denmark. \\
e-mail: mll@math.au.dk

\end{document}